%% file: ClassSimpLinFunc2.tex
\documentclass[12pt]{amsart}

\usepackage{srcltx}  

\usepackage{geometry} 
\usepackage{amsmath}
\usepackage{amsfonts,amssymb,amsthm,amscd,latexsym,euscript}

\usepackage[all]{xy}
\SelectTips{cm}{10}

\input{symbols.tex}

\title[A classification of small linear functors]{A classification of small linear functors}
\author{Boris Chorny}

\date{\today} 

\address{\newline B. Chorny\newline
Department of Mathematics\newline
University of Haifa at Oranim\newline
Tivon, Israel
}

\email{chorny@math.haifa.ac.il}

\begin{document}

\begin{abstract}
We extend Goodwillie's classification of finitary linear functors to arbitrary small functors. That is we show that every small linear simplicial functor from spectra to pointed simplicial sets is weakly equivalent to a filtered colimit of representable functors represented in cofibrant spectra. Moreover, we present this classification as a Quillen equivalence of the category of small functors from spectra to simplicial sets equipped with the linear model structure and the opposite of the pro-category of spectra with the strict model structure.
\end{abstract}

\maketitle

\section{Introduction}
Calculus of homotopy functors is a method pioneered by T.~Goodwillie, \cite{Goo:calc1,Goo:calc2, Goo:calc3}, to decompose any homotopy functor $F\colon \cat C\to \cat D$ between two sufficiently nice model categories into a `Taylor tower' $F\to \ldots P_nF\to P_{n-1}F \to \ldots P_0F$, where $F\to P_nF$ is the universal `polynomial' approximation of $F$. For nice $F$ and sufficiently highly connected $X\in \cat C$ the tower `converges' in the sense that there is a weak  equivalence $F(X)\we \holim_n P_nF(X)$. 

The advantage of this representation of a functor $F$ is that the homotopy fibers $D_nF = \hofib (P_nF\to P_{n-1}F)$, called the homogeneous layers of the tower, have a relatively simple description. For a homotopy functor $F\colon \textup{Top}_\ast \to \textup{Top}_\ast$ of based spaces, it suffices to give one spectrum $\partial_nF$ with an action of $\Sigma_n$, called the `$n$th derivative' of $F$ at $\ast$, to recover the value of $D_nF$ in all finite CW-complexes $X$: 
\begin{equation}\label{Goodwillie-classification}
D_nF(X) \simeq \Omega^\infty(\partial_n F \wedge_{h\Sigma_n} X^{\wedge n}).
\end{equation} 
This feature allows for computational applications of Goodwillie's calculus, \cite{Arone-Mahowald, Dundas-Goodwillie-McCarthy}.

Since the machinery of homotopy calculus has proven to be an efficient computational tool, many authors keep developing the machine itself, to apply it in new setup. The question of generalization of Goodwillie's calculus to arbitrary model categories satisfying some reasonable assumptions was addressed recently by L.~A.~Pereira, \cite{Pereira}, G.~Biedermann and O.~R\"ondigs, \cite{Biedermann-Roendigs}, D.~Barnes and R.~Eldred, \cite{Barnes-Eldred}. The question of recovering the original functor $F$ from the symmetric sequence of its derivatives with some additional structure is addressed in the recent work of G.~Arone and M.~Ching, \cite{Arone-Ching-classification, Arone-Ching-crosseffects}.

In this paper we address a question requiring an extension of the calculus technique: what are the values of the homogeneous layers of the Taylor tower in infinite spaces $X$?

Goodwillie has resolved this problem in \cite{Goo:calc3} by assuming additionally that the functor $F$ is \emph{finitary}, i.e., $F$ commutes with filtered homotopy colimits, up to homotopy, so that its values in infinite spaces are determined, up to a weak equivalence, by the values of $F$ in finite spaces.

Under this assumption, the formula (\ref{Goodwillie-classification}) provides a complete classification of homogeneous functors. 1-homogeneous functors are also called \emph{linear}. More explicitly, a functor $F\colon \cat C \to \cat D$ of two model categories is called linear if it takes homotopy pushout squares to homotopy pullback squares and $F$ is reduced,i.e., $F(\ast)\simeq\ast$. 

Let us denote by $\cal S$ the category of pointed simplicial sets, also called spaces; $\Sp$ is the (Bousfield-Friedlander) category of spectra considered as a simplicial model category.

Goodwillie's classification tells us that if $F\colon \cal S\to \cal S$ is a finitary linear functor, then $F(X)\simeq \Omega^\infty(\partial_1F \wedge X)$ for all  $X\in \cal S$. It may be carried over to a classification of finitary linear functors $G\colon \Sp\to \cal S$ as $G(Y)\simeq \Omega^\infty(\partial_1F \wedge Y)$ for all cofibrant $Y\in \Sp$. We suggest to extend this  classification further, so that it would  apply to all small linear functors, i.e., commuting with $\lambda$-filtered colimits for some cardinal $\lambda$. The category of small simplicial functors from spectra to spaces is denoted by $\cal S^\Sp$. For any (large) simplicial category $\cat E$, the category of small functors, $\cal S^{\cat E}$ is a locally small category; for basic model categorical techniques in such categories see \cite{Chorny-Dwyer},\cite{Chorny-Rosicky}; for category-theoretical foundations see \cite{Day-Lack, Chorny-Rosicky-class-accessible}. 

The basic building blocks of our classification are the representable functors $R^E(Y)=\hom(E,Y)$, for some cofibrant $Y\in \Sp$. To turn $R^E$ into a linear functor the reader may precompose it with a fibrant replacement in spectra. We have found a more elegant way to make this functor homotopy meaningful: consider the fibrant-projective model structure on $\cal S^\Sp$ introduced in \cite{Duality} (fibrations and weak equivalences are the natural transformations inducing fibrations and weak equivalences respectively, between the values of the functors in fibrant objects), so that every representable functor is (fibrant-projectively) weakly equivalent to a proper linear functor. Notice that the representable functors are not finitary, unless represented in compact spectra. If we consider a filtered colimit of representable functors, then we obtain a linear functor again, since filtered colimits commute with homotopy pullbacks in the category of spaces. The main result of this paper is that these are all the possibilities, every linear functor is (fibrant-projectively) weakly equivalent to a filtered colimit of representable functors. We formulate this result for a general combinatorial proper simplicial stable model category \cat E with a view towards a general classification of $n$-homogeneous small simplicial functors.

\begin{theorem_nonum}[{\bf\ref{classification}}]
Let $F\in \cal S^{\cat E}$ be a linear functor. Then there exists a
filtered diagram $J$ and a functor $G=\colim_{j\in J}R^{X_j}$ with
cofibrant $X_j\in \cat E$ for all $j\in J$ and a weak equivalence
$f\colon \tilde F\to G$ for some cellular approximation $\tilde
F\we F$ in the fibrant-projective model structure.
\end{theorem_nonum}

This classification can be viewed as a Quillen-equivalence of the linear model structure on the category of small functors and the opposite of the category \proE, provided that the linear model structure on $\cal S^{\cat E}$ exists. The linear model structure is a (left Bousfield) localization of the fibrant-projective model structure, and we managed to construct it under an additional assumption that all objects in \cat E are cofibrant.

A similar result was obtained in \cite{Chorny-ClassHomFun} for small spectral functors from spectra to spectra, though we claimed the classification of all small spectral homotopy functors in that paper. The reason is that every small spectral homotopy functor is linear, hence there is no interesting calculus theory for spectral functors. 

The paper is organized as follows. 

In Section~\ref{preliminaries} we recall basic facts about pro-categories and construct a left adjoint to the pro-representable functor, which allows us to view the opposite of the pro-category as a reflective subcategory of the category of small simplicial presheaves. If \cat E is a proper combinatorial simplicial model category, then we show that this is a Quillen adjunction, provided that the category of small functors is equipped with the fibrant-projective model structure and $\proE$ is equipped with the strict model structure.

In Section~\ref{linear-model} we construct a (non-functorial) localization $Q$ in $\cal S^\cat E$ such that the local objects are precisely the fibrant linear functors. Note that the construction of $Q=P_1$ is performed by a completely model-categorical technique and does not rely on Goodwillie's ideas, using the fat small-object argument instead \cite{fat}. Similar construction is possible for an arbitrary polynomial approximation $P_n$. If all objects in \cat E are cofibrant, then we are able to construct the $Q$-localization of $\cal S^\cat E$ in order to obtain the linear model structure. We do use a generalization of Goodwillie's technique due to L.A.~Pereira, \cite{Pereira} in this argument.

In Section~\ref{section-classification} we prove our main classification result. And finally in Section~\ref{section-Quillen-equivalence} we present our classification as a Quillen equivalence, provided that the linear model structure was established on the category of small functors. The Appendix is devoted to the proof of a technical lemma establishing the conditions sufficient for a reflection, which is also a Quillen adjunction to be a Quillen equivalence. It seems useful to separate this lemma for future reference, since similar considerations were used in recent papers, \cite{Duality, Chorny-ClassHomFun, Chorny-BrownRep}.

\section{Preliminaries on pro-categories}\label{preliminaries}
In this section and all over the paper we use the adjective `simplicial' to denote enrichment over the \emph{pointed} simplicial sets, $\cal S$, though the results of this section apply to the unpointed simplicial sets as well.

Let $\cat E$ be a complete and cocomplete simplicial category. That means the hom-sets of \cat E are pointed simplicial sets, and in addition \cat E is tensored and cotensored over the pointed  simplicial sets. The goal of this preliminary section is to show that the opposite of the category pro-\cat E is equivalent to a reflective subcategory of small functors from \cat E to spaces. If we work with the fibrant-projective model structure on the category of small functors \cite[3.6]{Duality}, this adjunction is a Quillen pair and it carries over to the level of the homotopy categories.

We start from a simpler adjunction:
\begin{equation}\label{simpler-adjunction}
\xymatrix{
Z\colon\cal S^{\cat E}
\ar@/^10pt/@{..>}[r] & \cat E^{\op}:\!Y,
				\ar@{_(->}[l]
}
\end{equation}
where $Y(E) = R^E(-)= \hom_{\cat E}(E,-)$ and $Z(F) = \hom(F, \Id_{\cat E}) = \int_{E\in \cat E}\hom_{\cat E}(F(E),E)$. The left adjoint $Z(F)$ exists since all $F\in \cal S^{\cat E}$ are small functors. The adjointness is readily verified:
\begin{multline*}
\hom_{{\cat E}^{\op}}(Z(F), E) = \hom_{{\cat E}^{\op}}(\hom(F, \Id_{{\cat E}}), E)=\hom_{{\cat E}}(E, \hom(F, \Id_{{\cat E}}))=\\
\hom_{\cal S^{{\cat E}}}(E\otimes F, \Id_{{\cat E}}) = \hom_{\cal S^{{\cat E}}}(F, \Id_{{\cat E}}^E) = \hom_{\cal S^{{\cat E}}}(F, R^E) = \hom_{\cal S^{{\cat E}}}(F, Y(E)).
\end{multline*}

Let us recall the definition of the pro-category for a given category $\cat E$. The objects of the category pro-\cat E are cofiltered diagrams of objects of \cat E, i.e., for every filtering $I$, any functor $X\colon I^{\op}\to \cat E$ is a pro-object. We denote this pro-object as $X_{\bullet} = \{X_i\}_{i\in I}$.

A morphism $f\colon \{X_i\}_{i\in I}\to \{Y_j\}_{j\in J}$ between two pro-objects is given by a function $\varphi\colon \obj J\to \obj I$ and a morphism in \cat E $f\colon X_{\varphi(j)}\to Y_j$ for all $j\in J$ such that for all $j_1,j_2\in J$ there exists $i\in I$ with maps $\iota_1\colon \varphi(j_1) \to i$ and $\iota_2\colon \varphi(j_2)\to i$ the diagram
\[
\xymatrix{
X_i
\ar[r]^{X(\iota_{1})}
\ar[d]_{X(\iota_{2})} & X_{\varphi(j_1)}
						\ar[r]^{f_{j_1}} & Y_{j_1}
										\ar[d]^{}\\
X_{\varphi(j_2)}
\ar[rr]_{f_{j_2}} & &  Y_{j_2}
}
\]
commutes. Formally,
\[
\hom_{\proE}(\{X_i\}, \{Y_j\}) = \lim_{j\in J}\colim_{i\in I}\hom_{\cat E}(X_i, Y_j).
\]
The category of pro-objects in \cat E is enriched over the category of simplicial sets with the simplicial $\hom_{\proE} (-, -)$ calculated by the above rule, while taking $\hom_{\cat E}(-, -)$ to be the simplicial $\hom$-functor in the category \cat E.

The category of small functors $\cal S^{\cat E}$ consists of small functors as objects and natural transformations as morphisms. We recall that a functor $F\colon \cat E\to \cal S$ is small if it is a left Kan extension of its restriction to some small subcategory; equivalently, small functors are small weighted colimits of representable functors.

The restriction of the Yoneda embedding $Y\colon (\proE)^{\op} \to \cal S^{\proE}$ to the subcategory $\cat E\overset c \cofib \proE$ is a functor $P\colon (\proE)^{\op} \to \cal S^{\cat E}$ defined as a composition $P=c^\ast Y$ and sending every pro-object $X_{\bullet}$ into the pro-representable functor $\hom_{\proE}(X_{\bullet}, - )\colon \cat E \to \cal S$. By the definition of morphisms in the category \proE, the pro-representable functor $\hom_{\proE}(\{X_i\}, - )=\colim_{i\in I}\hom_{\cat E}(X_i,-)$ is a filtered colimit of representable functors $R^{X_{i}}$ over $I$. In particular, every pro-representable functor is small.

Our goal in this section is to show that the functor $P$ has a left adjoint. 

\begin{proposition}\label{adjunction}
The functor $P\colon (\proE)^{\op} \to \cal S^{\cat E}$ has a left adjoint $O\colon \cal S^{\cat E}\to (\proE)^{\op}$.
\end{proposition}
\begin{proof}
We shall use the simpler adjunction (\ref{simpler-adjunction}) constructed above, Freyd's adjoint functor theorem and the fact that the category of small simplicial functors is class-finitely presentable \cite[2.2]{Chorny-Rosicky}.

Let $\cal S^{\cat E}\ni F = \colim_{i\in I} C_i$, where $C_i$ is a finitely presentable objects in $\cal S^{\cat E}$ for all $i$. Then,
\begin{multline*}
\hom_{\cal S^{\cat E}}(F, PX_{\bullet}) = \hom_{\cal S^{\cat E}}(\colim_{i\in I} C_i, \colim_{j\in J}R^{X_j}) = \lim_{i\in I}\hom_{\cal S^{\cat E}}(C_i, \colim_{j\in J}R^{X_j}) = \\
\lim_{i\in I}\colim_{j\in J}\hom_{\cal S^{\cat E}}(C_i, R^{X_j}) = \lim_{i\in I}\colim_{j\in J}\hom_{\cat E^{\op}}(ZC_i, X_j) = \\
\lim_{i\in I}\colim_{j\in J}\hom_{\cat E}(X_j, ZC_i) = \hom_{\proE}(\{X_j\},\{ZC_i\}) = \hom_{(\proE)^{\op}}(\{ZC_i\}, \{X_j\}).
\end{multline*}

The representation of $F$ as a filtered colimit of compact objects is not unique, but if we take any representation of this kind $F=\colim_{i\in I}C_i$, then the map $f\colon F\to\colim_{i\in I}R^{ZC_i} = P\{ZC_i\}$ is as a solution set, since, according to the computation above, every map $F\to PX_{\bullet}$ factors through $f$. Freyd's adjoint functor theorem
implies the existence of the left adjoint for $P$, and we can
compute its value, up to an isomorphism, by choosing a
representation for $F$ and assigning $OF = \{ZC_i\}_{i\in I}$.
\end{proof}

Let us assume now that $\cat E$ is a (stable) proper, combinatorial, simplicial model category. (Stability is not used in the following proposition, but will be used later). Then the category of small functors $\cal S^\cat E$ may be equipped with the fibrant-projective model structure constructed in \cite[3.6]{Duality}. Fibrant-projective weak equivalences and fibrations are the natural transformations of functors inducing levelwise weak equivalences or fibrations between their values in fibrant objects. The category of pro-objects in \cat E  may be equipped, in turn, with the strict model structure, \cite{Isaksen-strict}, where a map of pro-object is a weak equivalence or a cofibration if it is an essentially levelwise weak equivalence or an essentially levelwise cofibration. 

We conclude the categorical preliminaries by the following proposition that states, essentially, that the opposite of the homotopy category of $\proE$ is a co-reflective subcategory of the homotopy category of small functors with the fibrant-projective model structure.

\begin{proposition}\label{Quillen-map}
The pair of adjoint functors
\[
\xymatrix{
P\colon (\proE)^{\op} \ar@/^/[r] & \cal S^{\cat E} : \! O, \ar@/^/[l]
}
\]
constructed in Proposition~\ref{adjunction} is a Quillen pair if the category of small functors $\cal S^{\cat E}$ is equipped with the fibrant-projective model structure and the category $\proE$ is equipped with the strict model structure.
\end{proposition}
\begin{proof}
It suffices to show that the right adjoint $P$ preserves fibrations and trivial fibrations.

Consider a trivial fibration or a fibration $f^{\op}\colon Y_{\bullet}\to X_{\bullet}$ in  $(\proE)^{\op}$, i.e., $f\colon X_{\bullet}\to Y_{\bullet}$ is a
trivial cofibration or a cofibration in the strict model
structure on \proE, which means $f$ is an essentially levelwise trivial
cofibration or an essentially levelwise cofibration, where `essentially' means `up to reindexing'.

Let $f_{i}\colon X_{i}\to Y_{i}$, $i\in I$ be a levelwise trivial
cofibration or a levelwise cofibration representing $f$. Recall
that $PX_{\bullet}=\colim_{i\in I}R^{X_i}$,
$PY_{\bullet}=\colim_{i\in I}R^{Y_i}$. Then $Pf\colon
\colim_{i\in I}R^{X_i}\to \colim_{i\in I}R^{Y_i}$ is a trivial
fibration or a fibration, respectively, in the fibrant-projective
model structure, since each $f_{i}$ induces a trivial fibration or
a fibration of representable functors in the fibrant-projective
model structure, and filtered colimits preserve levelwise trivial
fibrations and fibrations.
\end{proof}

\section{Construction of the linear model structure}\label{linear-model}
The main objective of our work is to classify linear functors from a stable simplicial combinatorial model category $\cat E$ to simplicial sets, up to homotopy. The most convenient way to do so is to define a linear model category structure on the category of small functors with fibrant objects being exactly the fibrant (i.e., assuming fibrant values in fibrant objects of \cat E) linear functors, and to find a familiar Quillen equivalent model for this category. This section is devoted to construction of the linear model structure.

In the previous work, \cite{BCR}, the linear model category of functors from spaces to spaces was constructed in two stages. First we defined the homotopy model structure on  $\cal S^{\cal S}$. Next, we used the technique of Goodwillie calculus to further localize the homotopy model structure and obtain the linear model structure. Goodwillie's construction was extended by Luis Pereira, \cite{Pereira}, to more general model categories, so we could proceed by the same route, but we prefer to give a straightforward construction of the linear model structure, also introducing an alternative method allowing for construction of Goodwillie's polynomial approximation in fairly arbitrary model categories.

We start from the fibrant-projective model structure on the category of small functors (i.e., weak equivalences and fibrations are levelwise in fibrant objects), since we need to compare our model category with pro-\cat E (see Section~\ref{preliminaries}). But the general technique of localization we are about to describe may be applied to the projective model structure also, as to any class-cofibrantly generated model structure on the category of small functors.

Fibrant linear functors are the local objects with respect to the following class of maps:
\[
\cal F = \left\{\hocolim(R^{B}\leftarrow R^D\to  R^C)\to R^A \,\left|\, 
\vcenter{
   \xymatrix{A
          \ar[r]
          \ar[d] & B
                   \ar[d]\\
          C
          \ar[r] & D
   }
} \text{ homotopy pushout in } \cat E \right.\right\}.
\]

We recall a few basic definitions for the reader's convenience. 

A fibrant object $W$ is called $\cal F$-\emph{local} if for all $f\in \cal F$ the induced map $\mor{f,W}$ is a weak equivalences of simplicial sets. As observed in \cite{Dwyer_localizations}, it is an easy computation based on Yoneda's lemma to show that fibrant linear functors are precisely the $\cal F$-local objects in $\cal S^{\cat E}$.

A map $f\colon F\to G$ is an $\cal F$-\emph{equivalence} if for every cofibrant replacement $\tilde f\to f$ and for every  $\cal F$-local functor $W$ the induced map $\hom(\tilde f, W)$ is  a  weak equivalence of simplicial sets.

The class of generating trivial cofibrations for the fibrant-projective model structure is
\[
\cal J = \{R^{A}\otimes K\cofib R^{A}\otimes L \,|\, A\in \cat E \text{ fibrant; } K\trivcofib L \text{ generating triv. cofibration in } \cal S\}.
\]

\subsection*{Construction of linear approximation}

Our construction relies on the recently developed fat small object argument \cite{fat}. Namely we use \cite[Corollary~5.1]{fat} stating in particular that in a $\kappa$-combinatorial model category every cofibrant object is a $\kappa$-filtered colimit of $\kappa$-presentable cofibrant objects. 

Here we collect all the assumptions about the model category \cat E necessary for constructing localizations in the functor category $\cal S^{\cat E}$. We assume throughout this section that \cat E is a simplicial, combinatorial model category. Since the range category $\cal S$ is strongly left proper (see \cite[Definition~4.6]{Dundas-Roendigs-Ostvaer} or \cite[Definition~3.5]{Duality})  by \cite[Theorem~3.6]{Duality} the functor category $\cal S^{\cat E}$ with the fibrant-projective model structure is left proper; it is also right proper, since \cal S is right proper. Let us fix a cardinal $\kappa$ such that the model category $\cat E$ is $\kappa$-combinatorial, i.e., the domains and the codomains of the generating (trivial) cofibrations are $\kappa$-presentable, and the class of weak equivalences is a $\kappa$-accessible subcategory of the category of morphisms of $\cat E$. Since $\cat E$ is an accessible category, every small functor $F\in \cal S^{\cat E}$ is $\mu$-accessible for some cardinal $\mu$. We do not require in this section that \cat E be a stable model category. We will need this assumption for the classification theorem only. 

\begin{definition}
Let $F\in \cal S^\cat E$ be a small functor of accessibility rank $\mu$. Put $\lambda_F = \max\{\kappa, \mu\}^+\rhd \max\{\kappa, \mu\}$, and denote by $\cat E_{\lambda_F}\subset \cat E$ the \emph{subset} of $\lambda_F$-presentable objects. We define
\[
\cal F_{\lambda_F} = \left\{\hocolim(R^{B}\leftarrow R^D\to  R^C)\to R^A \,\left|\, 
\vcenter{
   \xymatrix{A
          \ar[r]
          \ar[d] & B
                   \ar[d]\\
          C
          \ar[r] & D
   }
} \text{ homotopy pushout in } \cat E_{\lambda_F} \right.\right\}.
\]
and
\[
\cal J_{\lambda_F} = \{R^{A}\otimes K\cofib R^{A}\otimes L \,|\, A\in \cat E_{\lambda_F} \text{ fibrant; } K\trivcofib L \text{ generating triv. cofibration in } \cal S\}.
\]
\end{definition}

\begin{remark}
We have to choose the successor cardinal $\max\{\kappa, \mu\}^+$ to ensure that the subcategory of weak equivalences in $\cat E$ is still $\lambda_F$-accessible.
\end{remark}

\begin{proposition}\label{Q-construction}
For any $F\in \cal S^{\cat E}$ there exists an $\cal F$-equivalence $\eta_F\colon F\to QF$, such that $QF$ is a fibrant $\cal F$-local functor. 
\end{proposition}
\begin{proof}
We form the set of horns on $\cal F_{\lambda_F}$ by first
replacing every map in $\cal F_{\lambda_F}$ with a cofibration,
obtaining the set $\tilde{\cal F}_{\lambda_F}$, and then forming a
box product with every generating cofibration in $\cal S$:
\[
\Hor(\cal F_{\lambda_F})=\{
A\otimes \Delta^n \coprod_{A\otimes \partial\Delta^n} B\otimes \partial\Delta^n \to B\otimes \Delta^n \,|\,
(A\cofib B) \in \tilde{\cal F}_{\lambda_F} \text{ and } n\geq 0
\}
\]

A simple adjunction argument implies (see e.g. \cite{Hirschhorn}) that if a fibration $X\to \ast$ has the right lifting property with respect to $\Hor(\cal F_{\lambda_F})$, $X$ is $\cal F_{\lambda_F}$-local, and therefore, to construct a localization of a small functor $F\in \cal S^\cat E$ with respect to $\cal F_{\lambda_F}$, it suffices to apply the small object argument for the map $F\to \ast$ with respect to the set $\cal L= \Hor(\cal F_{\lambda_F})\cup \cal J_{\lambda_F}$. We obtain a factorization $F\cofib Q(F)\fibr \ast$, where the cofibration is an $\cal L$-cellular map and the fibration has the right lifting property with respect to $\cal K$.

We omit the standard verification based on the left properness of $\cal S^\cat E$, \cite[Section~4]{Duality}, that the cofibration $\eta_F\colon F\cofib QF$ is an $\cal F_{\lambda_F}$-equivalence, and conclude that $QF$ is a homotopy localization of $F$ with respect to $\cal F_{\lambda_F}$.

Notice that $QF$ is obtained as a colimit of $\lambda_F$-accessible functors, and therefore $QF$ is itself a $\lambda_F$-accessible functor. 

For any homotopy pushout square we compute the corresponding homotopy pushout square of cofibrant objects
\begin{equation}\label{cube}
\xymatrix{
\tilde A
\ar@{->>}[dr]^{\dir{~}}
\ar@{^(->}[rr]
\ar@{_(->}[dd]
&	&	\tilde{B}
	\ar@{->>}[dr]^{\dir{~}}
	\ar'[d][dd]\\
&	A
	\ar[rr]
	\ar[dd]
&	&	B
		\ar[dd]\\
\tilde{C}
\ar@{->>}[dr]^{\dir{~}}
\ar'[r][rr]
	&	&	\tilde{B}\coprod_{\tilde{A}}\tilde{C}
			\ar[dr]^{\dir{~}}		\\
&	C
	\ar[rr]
&		&D,
}
\end{equation}
obtaining a cofibrant object  $\tilde C\hookleftarrow \tilde A \hookrightarrow \tilde B$ in $\cat E^{\cdot \leftarrow \cdot \to \cdot}$ with the projective model structure. Left properness of \cat E implies, by the cube lemma, \cite[ 5.2.6]{Hovey}, that the induced map $\tilde{B}\coprod_{\tilde{A}}\tilde{C}\to D$ is a weak equivalence. Now we apply \cite[Corollary~5.1]{fat} to the class of cofibrant objects in the category $\cat E^{\cdot \leftarrow \cdot \to \cdot}$ with the projective model structure and obtain the diagram $\tilde C\hookleftarrow \tilde A \hookrightarrow \tilde B$ as a $\lambda_F$-filtered colimit of diagrams $\{C_i\hookleftarrow A_i\hookrightarrow B_i \}_{i\in I}$ with $A_i, B_i, C_i$ cofibrant objects of $\cat E_{\lambda_F}$ for all $i\in I$.

Since \cat E is $\lambda_F$-combinatorial, the $\lambda_F$-filtered colimits are homotopy colimits, hence the induced map $\colim_{i\in I}( C_i \coprod_{A_i} B_i)  \to \tilde{B}\coprod_{\tilde{A}}\tilde{C}$ is a weak equivalence.

Because $QF$ is a $\lambda_F$-accessible functor, and since $QF$ is $\cal F_{\lambda_F}$-local, it converts homotopy pushouts of $\lambda_F$-presentable objects into homotopy pullbacks. Since filtered colimits commute with homotopy pullbacks, we obtain

\begin{multline*}
QF\tilde A = QF(\colim_{i\in I}A_i) = \colim_{i\in I}QFA_i = \colim_{i\in I} \holim \left(QFC_i\rightarrow QF\left(C_i\coprod_{A_i} B_i\right) \leftarrow QFB_i\right) =\\
\holim \colim_{i\in I} \left(QFC_i\rightarrow QF\left(C_i\coprod_{A_i} B_i\right) \leftarrow QFB_i\right) =\\
\holim \left(QF(\colim_{i\in I} C_i)\rightarrow QF\colim_{i\in I} \left(C_i\coprod_{A_i} B_i\right) \leftarrow QF(\colim_{i\in I} B_i)\right) =\\
\holim \left(QF(\tilde C)\rightarrow QF\left(\tilde C\coprod_{\tilde A} \tilde B\right) \leftarrow QF(\tilde B)\right).
\end{multline*}
Moreover, the functor $QF$ preserves the slanted weak equivalences in (\ref{cube}), since each  of these maps may be presented as a $\lambda_F$-filtered colimit of weak equivalences of $\lambda_F$-presentable objects, and the latter are preserved by $QF$, since for every weak equivalence of $\lambda_F$-presentable objects $\varphi\colon X\we Y$, $QF$ converts homotopy pushouts of the form
\[
\xymatrix{
X
\ar[r]^{\dir{~}}_\varphi
\ar[d]_{\dir{~}}^\varphi
	& Y
	  \ar@{=}[d]\\
Y
\ar@{=}[r]
	& Y
}
\] to homotopy pullbacks.

Therefore $QF$ takes any homotopy pushout square to a homotopy pullback square. In other words, $QF$ is \cal F-local or linear. In addition, the coaugmentation map $\eta_F\colon F\cofib QF$ is an $\cal F_{\lambda_F}$-equivalence by construction, but every $\cal F$-local functor is also $\cal F_{\lambda_F}$-local, and hence every $\cal F_{\lambda_F}$-equivalence is also an $\cal F$-equivalence. We conclude that $QF$ is a (non-functorial) homotopy localization of $F$ with respect to $\cal F$.
\end{proof}

The construction $Q(-)$ depends on the accessibility rank of the entry functor, therefore fails to be functorial though it can easily be made functorial on any subcategory of functors of limited accessibility rank. We define separately what it does on maps. 

For every natural transformation of functors $f\colon F\to G$, we define $Qf$ as a lifting in the diagram
\[
\xymatrix{
F
\ar@{^(->}[d]_{\eta_F}
\ar[r]^f
      &  G
         \ar[r]^{\eta_G} & QG
                 \ar@{->>}[d]\\
QF
\ar[rr]
\ar@{-->}[urr]|{Qf}
       &  &  \ast
}
\]
The lift exists since the left vertical map is $\cal L$-cellular and the right vertical map
is $\cal L$-injective by construction.

Perhaps we will have to choose $Qf$ out of many maps that are
simplicially homotopic to each other, but the important property satisfied by
any of these choices is the commutativity of the square

\begin{equation} \label{A2}
\xymatrix{
F
\ar@{^(->}[d]_{\eta_F}
\ar[r]^f
      &  G
         \ar@{^(->}[d]^{\eta_G} \\
QF
\ar[r]_{Qf}
       & QG.
}
\end{equation}

The following proposition is a standard property of localization constructions.

\begin{proposition}\label{Q-equiv=F-equiv}
Let $f\colon F\to G$ be a natural transformation of two functors in $\cal S^{\cat E}$, then $Qf$ is a weak equivalence if and only if $f$ is an $\cal F$-equivalence. 
\end{proposition}
\begin{proof}
If $f$ is an $\cal F$-equivalence, then so is $Qf$ by the `2-out-of-3' property for $\cal F$-equivalences in the commutative square (\ref{A2}). Hence $Qf$ is an $\cal F$-local equivalence of $\cal F$-local objects, i.e., a weak equivalence by the $\cal F$-local Whitehead theorem, \cite[3.2.13]{Hirschhorn}.

Conversely, if $Qf$ is a weak equivalence, then $Qf$ is also an $\cal F$-equivalence and, by the `2-out-of-3' property, in (\ref{A2}) again, we conclude that $f$ is an $\cal F$-equivalence. 
\end{proof}

To establish the existence of the left Bousfield localization with respect to $Q$ we need to use an alternative construction of the linearization by L.~A.~Pereira, \cite{Pereira}. This is a generalization of Goodwillie's technique from \cite{Goo:calc3} and its advantage is that the construction uses only finite homotopy pullbacks and filtered colimits of functors and therefore commutes with finite homotopy limits by \cite[Proposition~4.10]{Pereira}, allowing for verification of the analog of Bousfield-Fridlander A6 condition. The disadvantage is that it applies only to homotopy functors.  

The first stage of application of Pereira's linearization is therefore a replacement of a functor by a homotopy functor. One way of doing it is to localize $\cal S^{\cat E}$ with respect to the class of maps $\{R^B\to R^A\, |\, A\we B \text{ weak equivalence in } \cat E\}$. See for example \cite{Chorny-ClassHomFun} for a similar construction when $\cat E = \Sp$. The problem with this approach is, that similarly to the construction of $Q$ we cannot make sure that the A6 condition is satisfied, unless the theory of equivariant nullifications is worked out, or the  theory of localizations in class-combinatorial model categories \cite{Chorny-Rosicky} is sufficiently developed. 

\subsection*{Assumption} To construct this localization only, we shall assume from now on that all objects of $\cat E$ are cofibrant. If this is not so, we can replace our combinatorial model category $\cat E$ with  a Quillen equivalent model in which all objects are cofibrant, \cite{Dugger-generation}.

\begin{proposition}\label{simplicial-functors}
Let \cat E be a simplicial model category, and $F\colon \cat E\to \cal S$ a simplicial functor. Suppose that $f,g\colon E_1\to E_2$ are simplicially homotopic maps in \cat E, then $Ff, Fg\colon FE_1 \to FE_2$ are simplicially homotopic maps in $\cal S$. 
\end{proposition}
\begin{proof}
Let $h\in\hom_{\cat E}(E_1\otimes J,E_2)=\hom_{\cat E}(E_1,E_2)^J$ be a simplicial homotopy between $f$ and $g$, where $J$ a generalized interval, \cite[9.5.5]{Hirschhorn}. Every simplicial functor $F$ is equipped with a natural map $\varepsilon\colon \hom_{\cat E}(E_1,E_2)\to \hom_{\cal S}(FE_1,FE_2)$.

Then the simplicial homotopy between $Ff$ and $Fg$ is $\varepsilon^J(h)\in\hom_{\cal S}(FE_1,FE_2)^J=\hom_{\cal S}(FE_1\otimes J,FE_2)$.
\end{proof}

\begin{corollary}\label{simplicial=homotopy}
Every simplicial functor $F\colon \cat E\to \cal S$ preserves weak equivalences between objects which are both fibrant and cofibrant.
\end{corollary}
\begin{proof}
Weak equivalences between fibrant and cofibrant objects are homotopy equivalences, hence by \cite[9.5.24(2)]{Hirschhorn}, simplicial homotopy equivalences. Proposition~\ref{simplicial-functors} implies then that any simplicial functor preserves simplicial homotopy equivalences and hence weak equivalences between fibrant and cofibrant objects.
\end{proof}

Let $\fib\colon \cat E \to \cat E$ be an accessible fibrant replacement functor equipped with a natural transformation $\eta \colon \Id_{\cat E}\to \fib$, then according to the corollary above, for any $F\in \cal S^{\cat E}$, there is a fibrant-projectively equivalent homotopy functor $F\eta\colon F\we F\circ \fib$. 

Recall that $Q$-fibrations are the maps with the right lifting properties with respect to the fibrant-projective cofibrations that are also $Q$-equivalences.

\begin{theorem}\label{Bousfield-localization}
Let \cat E be a combinatorial simplicial model category such that all objects of \cat E  are cofibrant. Then there exists the left Bousfield localization of the fibrant-projective model structure with respect to $Q$. In other words, there exists a model structure on the category of small functors $\cal S^{\cat E}$ with weak-equivalences being the $Q$-equivalences and the fibrations being the $Q$-fibrations. 
\end{theorem}
\begin{proof}
We will use the non-functorial version of the Bousfiled-Friedlander localization technique developed in \cite[Appendix~A]{Duality}. We have constructed so far a non-functorial localization construction $Q$ in Proposition~\ref{Q-construction}. To complete the proof of the theorem we need to verify the conditions A.2-6 of \cite[Theorem~A.8]{Duality}. 

Condition A.2 was verified in the construction of $Q$, see diagram \ref{A2}. Conditions A.3 and A.4 follow from Proposition~\ref{Q-equiv=F-equiv}, since $\cal F$-equivalences are closed under retracts and satisfy the `2-out-of-3' property.  

Condition A.5 is verified by choosing a cardinal $\lambda$ sharply bigger than the accessibility ranks of the functors constituting a commutative square, and then applying the functorial small object argument with respect to $\Hor(\cal F_{\lambda})\cup \cal J_\lambda$.

Condition A.6 relies on our assumption that all objects of \cat E are cofibrant, but we believe  it is possible to prove it more generally. Given a pullback square 
\[\xymatrix{
W
\ar[r]
\ar[d]_g & X
		\ar[d]^f \\
Z
\ar[r]_h & Y
}\]
with $h$ a $Q$-fibration and $f$ a $Q$-equivalence, if we precompose each functor with the fibrant approximation functor $\fib\colon\cat E\to \cat E$, then we obtain a homotopy pullback square of homotopy functors by Corollary~\ref{simplicial=homotopy} and by assuming that all objects in \cat E are cofibrant. Notice that precomposition with $\fib$ produces a weakly equivalent functor in the fibrant-projective model structure. Applying \cite[Proposition~4.10]{Pereira}, we obtain a homotopy pullback square
\[\xymatrix{
P_1\circ W\circ\fib
\ar[r]
\ar[d]_{P_1(g\circ \fib)} & P_1\circ X\circ\fib
		\ar[d]^{P_1(f\circ \fib)} \\
P_1\circ Z\circ \fib
\ar[r] & P_1\circ Y\circ\fib,
}\]
in which $P_1(f\circ \fib)$ is a weak equivalence, hence its base change $P_1(g\circ \fib)$ is a weak equivalence too. Therefore $g$ is an $\cal F$-equivalence and, by Proposition~\ref{Q-equiv=F-equiv}, $g$ is also a $Q$-equivalence.

This completes the required verification and we conclude that if all objects of $\cat E$ are cofibrant, then $\cal S^\cat E$ may be equipped with the $Q$-local model structure.
\end{proof}

\section{Classification of small linear functors}\label{section-classification}

In this section, we present a classification of small linear functors from a stable (= pointed model category such that its homotopy category is triangulated) combinatorial model category $\cat E$ to pointed simplicial sets $\cal S$. These are the small functors taking homotopy pushouts (=homotopy pullbacks) to homotopy pullbacks. Since every small functor $F\in \cal S^{\cat E}$ is a weighted colimit of representable functors, it sends  the zero object of \cat E to the one-point space (i.e. every small functor is reduced in the pointed situation). 

Let $\cal F$ be the class of maps ensuring that $\cal F$-local objects are precisely the fibrant linear functors.
Namely,
\[
\cal F=\left\{\left.
\hocolim\left(
\vcenter{
\xymatrix@=10pt{
R^{D}
\ar[r]
\ar[d]  & R^{B}\\
R^{C}
}
}
\right)
\longrightarrow R^{A}
\right|
\vcenter{
\xymatrix@=10pt{
A
\ar[r]
\ar[d] & B
             \ar[d]\\
C
\ar[r]  &  D
}
}
\text{-- homotopy pullback in}\, \cat E
\right\}.
\]

Our goal is to show that every linear functor is (fibrant-projectively) weakly equivalent to a filtered colimit of functors represented in cofibrant objects, i.e. to an image of a cofibrant pro-spectrum under the restricted Yoneda embedding $P$ constructed in Section~\ref{preliminaries}.

We begin with the lemma stating that filtered colimits of representable functors are closed under filtered colimits. In other words, filtered colimits of filtered colimits of representable functors are again filtered colimits.

It follows from the fact that if $F$ is a filtered colimit of filtered colimits of representable functors, then $F=P(O(F))$, hence it is a pro-representable functor.

\begin{lemma}\label{filtered-colimit}
The full subcategory generated by the filtered colimits of representable functors is closed under filtered colimits in $\cal S^{\cat E}$.
Moreover, the subcategory of filtered colimits of functors represented in cofibrant objects is also closed
under filtered colimits.
\end{lemma}

The proof is identical to the proof of the same lemma for spectral functors \cite[Lemma~5.1]{Chorny-ClassHomFun}.

In the next lemma we show that any representable functor $R^{X}$ smashed with a finite space $A$ is $\cal F$-equivalent to a representable functor again, cf. \cite[Lemma~3.1]{Chorny-BrownRep}.

\begin{lemma}\label{dual-lemma}
Let $A\in \cal S$ be a finite space and $X$ is a fibrant object of $\cat E$. Then the functor $A\wedge R^{X}$ is $\cal F$-equivalent to $R^{\hom(A,X)}$. 
\end{lemma}
\begin{proof}
If $A=\partial \Delta^{n}_{+}$, then the proof proceeds by induction on $n$, cf. \cite[Lemma~3.1]{Chorny-BrownRep}. For $n=0$ the statement is trivial, since $\partial \Delta^0_+ = \ast$ and $\ast \wedge R^X = R^{\hom(\ast, X)} = R^0 = \ast$ this is a constant functor assuming the value $\ast$ in every object. 

Suppose for induction that $\partial \Delta^{n-1}_+\wedge R^{X} \overset  {\cal F} \simeq R^{\hom(\partial \Delta^{n-1}_+, X)}$, where $\overset  {\cal F} \simeq$ means $\cal F$-equivalent. Since for a fibrant $X\in \cat E$ the functor $R^X$ is cofibrant in the fibrant-projective model structure, the natural map $\Delta^{n-1}_+\wedge R^X \to R^{\hom(\Delta^{n-1}_+, X)}$ is an $\cal F$-equivalence, for the domain of this map is weakly equivalent to $R^X$ and the weak equivalence $\hom(\Delta^{n-1}_+,X)\we \hom(S^0, X) = X$ induces the $\cal F$-equivalence $R^{\hom(\Delta^{n-1}_+, X)}\Feq R^X$ coming from the homotopy pullback square
\[
\xymatrix{
\hom(\Delta^{n-1}_+, X)
\ar[r]^<<<{\dir{~}}
\ar[d]_{\dir{~}} & X
			\ar@{=}[d]\\
X
\ar@{=}[r]	& X.
}
\]
Therefore
\begin{multline*}
\partial \Delta^n_+ \wedge R^X \simeq \colim\left(
\vcenter{
\xymatrix{
\partial \Delta^{n-1}_+\wedge R^{X}
\ar@{^(->}[d]
\ar@{^(->}[r] 				& \Delta^{n-1}_+\wedge R^{X}
								\\
\Delta^{n-1}_+\wedge R^{X}
 					& 
}
}
\right) \Feq\\
\hocolim\left(
\vcenter{
\xymatrix{
R^{\hom(\partial \Delta^{n-1}_+, X)}
\ar[d]
\ar[r] 				&  R^{\hom(\Delta^{n-1}_+,X)}
								\\
 R^{\hom(\Delta^{n-1}_+,X)}
}
}
\right) \xrightarrow{\Feq}
R^{\hom({\Delta^{n-1}_+} \coprod_{\partial \Delta^{n-1}_+} {\Delta^{n-1}_+}, X
						)}\\
						\simeq R^{X^{\partial \Delta_+^n}}.
\end{multline*}
In general, a finite space $A$ is obtained by a finite sequence $A_0\to A_1\to \ldots A_m$, where $A_0= \ast$ and $A_m=A$, and
\[
\vcenter{
\xymatrix{
\partial \Delta^{n_i}_+
\ar@{^(->}[d]
\ar[r]			  & A_i
					\ar@{^(->}[d]\\
\Delta^{n_i}_+
\ar[r]			  & A_{i+1}
}
}
\]
is a (homotopy) pushout for all $0\leq i\leq m$.

We finish the proof by induction on $i$, for $0\leq i\leq m$. Again, for $i=0$ the statement is trivial, and assuming the statement for $i$, we conclude that it is also true for $i+1$ since 
\[
\hom \left(
\vcenter{
\xymatrix{
\partial \Delta^{n_i}_+
\ar@{^(->}[d]
\ar[r]			  & A_i
					\ar@{^(->}[d]\\
\Delta^{n_i}_+
\ar[r]			  & A_{i+1}
}
}, X \right)
\]
is a homotopy pullback square in $\cat E$, and hence
\begin{multline*}
A_{i+1}\wedge R^X \simeq \colim \left(
\vcenter{
\xymatrix{
\partial \Delta^{n_i}_+\wedge R^{X}
\ar@{^(->}[d]
\ar@{^(->}[r] 				& A_i\wedge R^{X}
								\\
\Delta^{n_i}_+\wedge R^{X}
}
}
\right) \Feq\\
\hocolim
\left(
\vcenter{
\xymatrix{
 R^{\hom(\partial \Delta^{n_i}_+,X)}
\ar@{^(->}[d]
\ar@{^(->}[r] 				& R^{\hom(A_i,X)}
								\\
 R^{\hom(\Delta^{n_i}_+,X)}
}
}
\right) \Feq R^{\hom(A_{i+1}, X)}
\end{multline*}
\end{proof}

\begin{theorem}\label{classification}\label{F-euiv-filtered}
Let $F\in \cal S^{\cat E}$ be a linear functor. Then there exists a
filtered diagram $J$ and a functor $G=\colim_{j\in J}R^{X_j}$ with
cofibrant $X_j\in \cat E$ for all $j\in J$ and a weak equivalence
$f\colon \tilde F\to G$ for some cellular approximation $\tilde
F\we F$ in the fibrant-projective model structure.
\end{theorem}
\begin{proof}
Since $F$ is a linear functor, it is also a homotopy functor, hence, by \cite[Proposition~5.8]{Duality}, there exists a cellular approximation $\tilde F\we F$ such that for some cardinal $\lambda$ there is a transfinite sequence
of functors $\tilde F = \colim_{i\leq \lambda}\tilde F_{i}$, and $\tilde F_{i}$
is obtained from $\tilde F_{i-1}$ by attaching a generating cofibration of the form $A\wedge R^{X}\cofib B\wedge R^{X}$ for every successor cardinal $i\leq \lambda$ and $\tilde F_{i}=\colim_{a<i}\tilde F_{a}$ for every limit ordinal $i\leq \lambda$. The cofibration $A\cofib B$ is a generating cofibration in $\cal S$ and therefore $A$ and $B$ are finite simplicial sets. Moreover, the representing object $X\in \cat E$ may be chosen to be fibrant and cofibrant, since $F$ is a homotopy functor.

By \cite[Lemma~3.3]{Chorny-BrownRep}, there exists a countable
sequence $\{F_{k}'\}_{k<\omega}$ such that $F'_{0}=\ast$,
$\tilde F=\colim_{k<\omega}F_{k}'$ and for each $k>0$ there is a pushout
square
\[
\xymatrix{
\displaystyle{\bigvee_{s\in S_{k-1}} A_{s}\wedge R^{ X_{s}} }
\ar[r]
\ar@{^{(}->}[d]         &F'_{k-1}
                    \ar[d]\\
\displaystyle{\bigvee_{s\in S_{k-1}} B_{s}\wedge R^{ X_s}}
\ar[r]           &F'_{k},\\
}
\]
where the coproduct is indexed by the subset $S_{k-1}\subset \lambda$ corresponding to the cells coming from various stages of the original sequence $\{\tilde F_{i}\}_{i\leq \lambda}$, such that their attachment maps factor through the $(k-1)$-st stage of the previously constructed sequence.

The coproduct of maps in the commutative square above is a filtered colimit of finite coproducts over the filtering $J_{k-1}$ of the finite subsets of $S_{k-1}$. Let us think of the constant object $F'_{k-1}$ as a filtered colimit of the constant diagrams
over the same filtering $J_{k-1}$. However, colimits over
$J_{k-1}$ commute with pushouts, and hence we obtain the
representation of $F'_{k}$ as a filtered colimit of pushouts of
the following form
\begin{equation}\label{pushout-square}
\xymatrix{
\displaystyle{\bigvee_{s\in S_{k-1,j}} A_{s}\wedge R^{ X_{s}} }
\ar[r]_<<<<{\varphi_{k-1,j}}
\ar@{^{(}->}[d]         &F'_{k-1}
                    \ar[d]\\
\displaystyle{\bigvee_{s\in S_{k-1,j}} B_{s}\wedge R^{ X_s}}
\ar[r]           &F_{k,j},\\
}
\end{equation}
where $S_{k-1,j}\subset S_{k-1}$ is a finite subset corresponding to the element $j\in J_{k-1}$.

Now, by Lemma~\ref{dual-lemma}, there are $\cal F$-equivalences in
the fibrant projective model category: $A_{s}\wedge R^{{X}_{s}}\Feq R^{\hom(A_s,{X}_{s})}$ and $B_{s}\wedge R^{{X}_{s}}\Feq R^{\hom(B_s,{X}_{s})}$. 

Moreover, any finite coproduct of representable functors represented in fibrant objects is $\cal F$-equivalent to a representable functor by an inductive argument on the number of terms that begins by observing that a coproduct of two representable functors $R^{U}\vee R^{V}$ is $\cal F$-equivalent to $R^{ U\times V}$, since the map $R^{ U}\vee R^{ V}\simeq\hocolim(R^{U}\leftarrow R^{0}\to R^{V})\longrightarrow R^{ U\times V}$ is an element in $\cal F$ corresponding  to the homotopy pullback square
\[
\xymatrix{
 U\times V
\ar[r]
\ar[d]                     &  U
                                 \ar@{->>}[d]\\
 V
\ar@{->>}[r]                    &  0.\\
}
\]

In other words, the entries on the left side of the pushout square (\ref{pushout-square})
are $\cal F$-equivalent to representable functors represented in fibrant objects of \cat E.

Suppose for induction that there is an $\cal F$-equivalence $F'_{k-1}\to \colim_{l\in L_{k-1}}R^{Y_l}$, where $L_{k-1}$ is a filtered category and the representing objects $Y_{l}\in \cat E$ are fibrant and cofibrant. Then we obtain a morphism of the pushout diagram (\ref{pushout-square}) into a commutative square that is also a homotopy pushout composed of filtered colimits of representable functors constructed as follows
\begin{equation}\label{dash-arrow}
\xymatrix{
R^{\left(\displaystyle{\prod_{s\in S_{k-1,j}}}\hom(A_s, {X_s})\right)_{\text{cof}}}
\ar[rrr]_{\psi_{k-1,j}}
\ar[ddd]        &  &  & \displaystyle{\colim_{l\in L'_{k-1}}} R^{Y_l}
                        \ar[ddd]\\
 & \displaystyle{\bigvee_{s\in S_{k-1,j}} A_{s}\wedge R^{X_{s}}}
\ar@{^{(}->}[d]
\ar[r]_<<<<{\varphi_{k-1,j}}
\ar[ul]             &F'_{k-1}
                    \ar[d]
                    \ar[ur]\\
 & \displaystyle{\bigvee_{s\in S_{k-1,j}} B_{s}\wedge R^{X_s}}
    \ar[r]
    \ar[dl]          & F_{k,j}
                        \ar@{-->}[dr]\\
R^{\left(\displaystyle{\prod_{s\in S_{k-1,j}}}\hom(B_s, {X_s})\right)_{\text{cof}}}
\ar[rrr]        &  &  & \displaystyle{\colim_{l\in L'_{k-1}}} R^{Y'_l}.\\
}
\end{equation}
The diagonal maps on the left are obtained as compositions of the unit of the adjunction (\ref{adjunction})
with a map induced by the cofibrant approximations in pro-$\cat E$:
\begin{eqnarray*}
\left(\displaystyle{\prod_{s\in S_{k-1,j}}}\hom(A_s,{X_s})\right)_{\text{cof}} \trivfibr \displaystyle{\prod_{s\in S_{k-1,j}}}\hom(A_s, {X_s}),\\
\left(\displaystyle{\prod_{s\in S_{k-1,j}}}\hom(B_s, {X_s})\right)_{\text{cof}} \trivfibr \displaystyle{\prod_{s\in S_{k-1,j}}}\hom(B_s, {X_s}).
\end{eqnarray*}

The universal property of the unit of adjunction guarantees the existence of a natural map
\[
R^{\displaystyle{\prod_{s\in S_{k-1,j}}
					\hom(A_s, X_s)
			    }
    } 
\longrightarrow \colim_{l\in L_{k-1}}R^{Y_l}.
\]
The corresponding map in the pro-category has a lift to the cofibrant replacement of the constant pro-spectrum,
since the pro-spectrum $\{Y_l\}_{l\in L_{k-1}}$ is (levelwise) cofibrant.
\[
\xymatrix{
 & \left(\displaystyle{\prod_{s\in S_{k-1,j}}}\hom(A_s,{X_s})\right)_{\text{cof}}
    \ar@{->>}[d]^{\dir{~}} \\
 \{Y_l\}_{l\in L_{k-1}}
 \ar[r]
 \ar@{-->}[ur] & \displaystyle{\prod_{s\in S_{k-1,j}}}\hom(A_s,{X_s})
}
\]
The source of the dashed map in the diagram above may be replaced
by an isomorphic pro-object $\{Y_l\}_{l\in L'_{k-1}}$ with a final
indexing subcategory $L'_{k-1}\subset L_{k-1}$, so that the
resulting map is reindexed into a natural transformation of
contravariant $L'_{k-1}$-diagrams with a constant diagram in the
target. The induced map in the category of functors is denoted by
$\psi_{k-1,j}$, and it factors through every stage of the
colimit. Thus the outer pushout diagram in (\ref{dash-arrow}) may
be viewed as a filtered colimit of pushout diagrams indexed by
$L'_{k-1}$.

Let us put 
\[Y'_l = Y_l \times_{\left(\displaystyle{\prod_{s\in S_{k-1,j}}}\hom(A_s, {X_s})\right)_{\text{cof}}}\left(\displaystyle{\prod_{s\in S_{k-1,j}}}\hom(B_s, {X_s})\right)_{\text{cof}}.
\]
This is a homotopy pullback in \cat E, and hence $R^{Y'_l}$ is $\cal F$-equivalent to the
homotopy pushout of the corresponding representable functors in
the fibrant-projective model structure on the category of small
functors $\cal S^\cat E$.

Taking the filtered colimit of these commutative squares indexed
by $L'_{k-1}$, we obtain the outer square of (\ref{dash-arrow}),
and since filtered colimits preserve both $\cal F$-equivalences (by an argument simililar to \cite[Lemma~1.2]{CaCho} and the fact that generating cofibrations in $\cal S^{\cat E}$ have finitely presentable domains and codomains) and homotopy pushouts, we conclude that $\displaystyle{\colim_{l\in L'_{k-1}}} R^{Y'_l}$ is $\cal F$-equivalent to the homotopy pushout of the outer square of (\ref{dash-arrow}). Therefore the dashed arrow in (\ref{dash-arrow}) is an $\cal
F$-equivalence. In other words, $F_{k,j}$ is $\cal F$-equivalent
to a filtered colimit of representable functors.

Therefore $F'_{k}=\colim_{j\in J_{k-1}} F_{k,j}$ is a filtered colimit of functors $\cal F$-equivalent
to filtered colimits of representable functors, which, in turn, is $\cal F$-equivalent to a filtered colimit of representable functors by Lemma~\ref{filtered-colimit}.

Finally, $F=\colim_{k<\omega}F'_{k}$ is a countable sequential colimit of filtered colimits
of functors $\cal F$-equivalent to representable functors, which may be reindexed into a single filtered colimit
of functors $\cal F$-equivalent to representable functors by Lemma~\ref{filtered-colimit}.
\end{proof}

\section{The classification presented as a Quillen equivalence}\label{section-Quillen-equivalence}

The localization of the fibrant-projective model structure constructed in Theorem~\ref{Bousfield-localization} applies only for simplicial combinatorial model categories \cat E, such that all objects in \cat E are cofibrant. On this assumption we prove a refinement of the above classification of linear functors presenting it as a Quillen equivalence with the opposite of the category of pro-objects in \cat E.

\begin{theorem}\label{main-theorem}
Let \cat E be a simplicial combinatorial stable model category. Suppose that all objects in \cat E are cofibrant, then the Quillen adjunction
$\xymatrix{O\colon \cal S^{\cat E}
                \ar@/_/[r]  &
                                (\text{\emph{pro-}}\cat E)^{\op} \!: P
                                \ar@/_/[l]}$
is a Quillen equivalence if  $\,\cal S^{\cat E}$ is equipped with the
linear model structure and \emph{pro-}\cat E is equipped with the strict
model structure.
\end{theorem}

The following lemmas precede the proof of this theorem.

\begin{lemma}\label{derived-counit}
Let $X_\bullet\in \proE$ be a cofibrant object. Then,
$PX_\bullet\in \cal S^{\cat E}$ is a filtered colimit of representable
functors and not necessarily cofibrant. Consider a cofibrant
replacement, $p\colon \widetilde{PX_\bullet} \trivfibr
PX_\bullet$. Then the left adjoint $O$ preserves this weak
equivalence: $Op\colon O\widetilde{PX_\bullet} \we OPX_\bullet$.
\end{lemma}
\begin{proof}
As the left Quillen functor $O$ preserves weak equivalences
between cofibrant objects,  it suffices to prove that $O$ takes
into a weak equivalence some cofibrant approximation of $P\{X_i\}
= \colim_i R^{X_i}$. Consider the cofibrant approximation:
$q\colon \hocolim_i R^{\hat X_i} = \overline{PX_\bullet}\to
PX_\bullet$, where $q$ is induced by the fibrant-projective
cofibrant approximations $R^{\hat X_i}\we R^{X_i}$, while the maps
$X_i \trivcofib \hat X_i$ are the functorial fibrant
approximations in $\cat E$. $\overline{PX_\bullet}$ is cofibrant as a
homotopy colimit of a diagram with cofibrant entries (we assume
here that a homotopy colimit is defined as a coend with a
projectively cofibrant, contractible diagram of spaces, i.e. a
left Quillen functor preserving cofibrant objects).  By
\cite[2.3(ii)]{Dugger-generation}, the map $q$ is a weak equivalence,
since filtered colimits in the class-cofibrantly generated
fibrant-projective model structure on $\cal S^{\cat E}$ are homotopy
colimits.

The left adjoint $O$ preserves colimits and homotopy colimits as a left Quillen
functor, and hence the map $Oq\colon O\overline{PX_\bullet}\to
OPX_\bullet$ is essentially the map $Oq\colon \hocolim OR^{\hat
X_i}\to \colim OR^{X_i}$, or just $Oq\colon \hocolim \hat X_i \to
\colim X_i$ in the opposite of the strict model structure on
$\proE$. However, the strict model structure on $\proE$
is class-fibrantly generated, \cite{pro-spaces}, and therefore the
dual model structure is class-cofibrantly generated and the map
$Oq$ is a weak equivalence by \cite[2.3(ii)]{Dugger-generation}.
Therefore, $Op$ is also a weak equivalence.
\end{proof}

\begin{lemma}\label{derived-unit}
The derived unit map $u_F\colon F\to P\widehat{OF}$ is a $Q$-equivalence for all cofibrant $F\in \cal S^{\cat E}$.
\end{lemma}
\begin{proof}
By Proposition~\ref{Q-equiv=F-equiv}, it suffices to check whether $u_F$ is
an $\cal F$-equivalence, i.e. it suffices to verify that
$\hom(\widetilde{u_F}, W)$ is a weak equivalence for any $\cal F$-local functor $W$. By
Theorem~\ref{classification}, $W$ is weakly equivalent to a
filtered colimit of representable functors represented in
cofibrant objects of \cat E, and hence $W\simeq PX_\bullet$ for some
cofibrant $X_\bullet\in \proE$.

By adjunction, the map $\hom(\widetilde{u_F}, PX_\bullet)$ is
naturally isomorphic to the map $\hom(O\widetilde{P\widehat{O
F}}, X_\bullet)\to \hom({OF, X_\bullet})$.

By Lemma~\ref{derived-counit}, $O\widetilde{P\widehat{OF}}\simeq OP\widehat{OF}= \widehat{OF}$, showing
that the last map is a weak equivalence.
\end{proof}

\begin{proof}[Proof of Theorem~\ref{main-theorem}]
To establish the required Quillen equivalence we will use Lemma~\ref{reflection-Quillen-equiv}. We need to verify two conditions:
\begin{enumerate}
\item For all cofibrant $F\in \cal S^{\cat E}$, the derived unit map $u_F\colon F\to P\widehat{OF}$ is a weak equivalence;
\item For all fibrant $X_{\bullet}\in \proE^{\op}$ and for any cofibrant replacement $p\colon \widetilde{PX_\bullet} \trivfibr PX_\bullet$, the map $Op\colon O\widetilde{PX_\bullet} \to OPX_\bullet$ is a weak equivalence.
\end{enumerate}

The first condition was verified in Lemma~\ref{derived-unit} and the second condition in Lemma~\ref{derived-counit}.
\end{proof}

\appendix

\section{Quillen equivalence with a reflective subcategory}\label{Quillen-equivalence}
This appendix is devoted to the proof of a general statement about Quillen equivalences of model categories. That is, we analyse in detail the situation where the right adjoint is a fully faithful embedding of categories and formulate sufficient conditions  to claim that in this particular case the adjunction is a Quillen equivalence.

\begin{lemma}\label{reflection-Quillen-equiv}
Let $\xymatrix{L\colon \cat M
                \ar@/_/[r]  &
                                \cat N \!: R
                                \ar@/_/[l]}$
be a Quillen pair. Suppose that the right adjoint $R$ is a fully faithful embedding of categories. Suppose in addition that 
\begin{enumerate}
\item For any cofibrant $M\in \cat M$, let $j\colon LM \trivcofib \widehat{LM}$ be a fibrant replacement in $\cat N$, then the derived unit map $M\xrightarrow{u} RLM \xrightarrow{Rj} R\widehat{LM}$ is a weak equivalence;
\item  For any fibrant $N\in \cat N$, let $p\colon \widetilde{RN}\trivfibr RN$ be a cofibrant replacement in $\cat N$,  then $Lp\colon L\widetilde{RN} \to LRN$ is a weak equivalence.
\end{enumerate}
Then the Quillen pair $(L,R)$ is a Quillen equivalence.
\end{lemma}
\begin{proof}
We need to show that for every cofibrant $M\in \cat M$ and for every fibrant $N\in \cat N$ a map $f\colon L(M)\to N$ is a weak equivalence in $\cat N$ if and only if the adjoint map $g\colon M\to RN$ is a weak equivalence in $\cat M$.

Suppose that $f\colon L(M)\to N$ is a weak equivalence. Applying a fibrant replacement on $L(M)$, we obtain a trivial cofibration $j\colon LM \trivcofib \widehat{LM}$ and a factorization of $f$ as $f=\hat f j$, where the lifting $\hat{f}$ exists since $N$ is fibrant. Moreover, $\hat{f}$ is a weak equivalence of fibrant objects by the `2-out-of-3' property. The adjoint map $g$ factors as a unit of the adjunction $u\colon M\to RLM$ composed with $Rf$: $g=R(f)u$, but $Rf = R(\hat{f}j) = R(\hat{f})Rj$, and hence $g = R(\hat{f}) (R(j)u)$. Now, $R(\hat{f})$ is a weak equivalence, since $R$ is a right Quillen functor and preserves weak equivalences of fibrant
objects. The composed map $R(j) u$ is a weak equivalence by the first condition of the lemma.

Conversely, suppose that $g\colon M\to RN$ is a weak equivalence. Consider a cofibrant replacement $p\colon \widetilde{RN}\trivfibr RN$. Then, there exists a lift $\tilde g\colon M\to \widetilde{RN}$ in the homotopy model structure. Note that $\tilde g$ is a weak equivalence of cofibrant objects by the `2-out-of-3' property, since $g=p\tilde g$. The adjoint map $f\colon LM \to N$ factors as $Lg$ followed by the counit $c\colon LRN\to N$, which is a natural isomorphism for all $N$, since $R$ is fully faithful. However, $Lg=LpL\tilde g$, where $Lp$ is a weak equivalence by the second condition of the lemma and $L\tilde g$ is a weak
equivalence, since $L$ is a left Quillen functor. Hence $f$ is a weak equivalence.
\end{proof}

\bibliographystyle{abbrv}
\bibliography{Xbib}

\end{document}

%% file: symbols.tex
\newcommand{\Sp}{\ensuremath{\textup{Sp}}}

\newcommand{\proE}{\ensuremath{\textup{pro-}{\cat E}}}

\newcommand{\op}{{\ensuremath{\textup{op}}}}

\DeclareMathOperator{\Hor}{\ensuremath{\textup{Hor}}}

\newcommand {\cofib} {\ensuremath{\hookrightarrow}}

\newcommand {\fibr} {\ensuremath{\twoheadrightarrow}}

\newcommand {\trivcofib} {\ensuremath{\tilde\hookrightarrow}}

\newcommand {\trivfibr} {\ensuremath{\tilde\twoheadrightarrow}}

\newcommand {\we} {\ensuremath{\tilde\rightarrow}}

\newcommand {\Feq} {\ensuremath{\overset {\cal F} \simeq}}

\DeclareMathOperator{\fib}{\textup{fib}}

\DeclareMathOperator{\hocolim}{\textup{hocolim}}
\DeclareMathOperator{\holim}{\textup{holim}}
\DeclareMathOperator{\hofib}{\textup{hofib}}

\newcommand{\cal}[1]{\ensuremath{\mathcal #1}} 

\newtheorem {theorem1}{Theorem}[section]
\newtheorem {theorem}[theorem1]{Theorem}
\newtheorem*{theorem_nonum}{Theorem}
\newtheorem {corollary}[theorem1]{Corollary}
\newtheorem {proposition}[theorem1]{Proposition}
\newtheorem {lemma}[theorem1]{Lemma}
\theoremstyle{definition}
\newtheorem {definition}[theorem1]{Definition}

\theoremstyle{remark}
\newtheorem {remark}[theorem1]{Remark}

\newcommand{\cat}[1]{\ensuremath{\EuScript #1}}

\newcommand{\colim}{\ensuremath{\mathop{\textup{colim}}}}

\newcommand{\Id}{\ensuremath{\textup{Id}}}

\newcommand{\obj}[1]{\ensuremath{\textup{obj}(#1)}}
\newcommand{\mor}[1]{\ensuremath{\textup{mor}(#1)}}

\renewcommand{\hom}{\ensuremath{{\rm hom}}}

\newcounter{zahl}%
    {\end{list}}%